\documentclass{amsart}

\usepackage{amssymb,stmaryrd}         
\usepackage[all]{xy}                              
\usepackage[colorlinks=true]{hyperref} 

\theoremstyle{plain}
\newtheorem{thm}{Theorem}[section]
\newtheorem{lem}[thm]{Lemma}
\newtheorem{cor}[thm]{Corollary}
\newtheorem{prop}[thm]{Proposition}

\theoremstyle{definition}
\newtheorem{ex}[thm]{Example}
\newtheorem{exs}[thm]{Examples}
\newtheorem{rem}[thm]{Remark}
\newtheorem{rems}[thm]{Remarks}
\newtheorem{dfn}[thm]{Definition}

\newcommand{\CC}{\mathbb{C}}       
\newcommand{\ZZ}{\mathbb{Z}}        
\newcommand{\QQ}{\mathbb{Q}}       
\newcommand{\NN}{\mathbb{N}}       
\newcommand{\kk}{k}                         

\newcommand{\lieh}{\mathfrak{h}}     
\newcommand{\lieg}{\mathfrak{g}}     

\newcommand{\ch}{\sp{\scriptscriptstyle\vee}}   

\newcommand{\JF}{\mathfrak{J}_F}

\newcommand{\FGL}{F}                  

\newcommand{\DF}{D_F}                    
\newcommand{\ADF}{\mathbf{D}_F}          

\newcommand{\Ruv}{R\llbracket u,v \rrbracket}

\newcommand{\RL}{R\llbracket \Lambda \rrbracket}
\newcommand{\RLF}{R\llbracket \Lambda \rrbracket_F}       
\newcommand{\Rhat}{\widehat{R}}
\newcommand{\Phireal}{\Phi^{re}}

\DeclareMathOperator{\im}{\mathrm{im}}    
\DeclareMathOperator{\GL}{GL}
\DeclareMathOperator{\Aut}{Aut}
\DeclareMathOperator{\rk}{rk}
\DeclareMathOperator{\Hom}{Hom}
\DeclareMathOperator{\ad}{ad}
\DeclareMathOperator{\id}{id}
\DeclareMathOperator{\cha}{char}
\DeclareMathOperator{\cork}{corank}
\DeclareMathOperator{\Span}{span}

\newcommand\sm{\smallskip}

\newcommand\lan{\langle} 
\newcommand\ran{\rangle}

\DeclareMathOperator{\rma}{A}

\DeclareMathOperator{\rmc}{C}

\newcommand\frh{\ensuremath{\mathfrak{h}}} 

\newcommand\al{\alpha}

\newcommand\be{\beta}

\newcommand\de{\delta} 
\newcommand\De{\Delta}

\newcommand\eps{\epsilon} 

\newcommand\la{\lambda} 
\newcommand\La{\Lambda}
\newcommand\om{\omega}

\title{The hyperbolic formal affine Demazure algebra}

\author{Marc-Antoine Leclerc}

\address{Department of Mathematics and Statistics, University of Ottawa, 585 King Edward, Ottawa ON K1N 6N5, Canada}
\email{mlecl084@uottawa.ca}
\urladdr{http://aix2.uottawa.ca/~mlecl084/} 

\thanks{This paper will be part of the author's PhD thesis written under the supervision of Erhard Neher and Kirill Zainoulline. It was partially supported by the NSERC Discovery grants of the supervisors.}



\begin{document}

\begin{abstract}
In the present paper we extend the construction of the formal (affine) Demazure algebra due to Hoffnung, Malag\'{o}n-L\'{o}pez, Savage and Zainoulline in two directions. First, we introduce and study the notion of a formal Demazure lattice in the Kac-Moody setting and show that all the definitions and properties of the formal (affine) Demazure operators and algebras hold for such lattices. Second, we show that for the hyperbolic formal group law the formal Demazure algebra is isomorphic (after extending the coefficients) to the Hecke algebra.
\end{abstract}

\maketitle

\section{Introduction}

A series of papers \cite{HMSZ}, \cite{CZZ}, \cite{CZZ1}, \cite{CZZ2} by Calm\`{e}s, Hoffnung, Malag\'{o}n-L\'{o}pez, Savage, Zainoulline and Zhong generalized the Kostant-Kumar \cite{KK}, \cite{KK1} nil-Hecke approach to equivariant cohomology of flag varieties to the context of algebraic oriented theories in the sense of Levine-Morel, with the respective formal group laws and finite root systems. Namely, they introduced an algebra $\RLF$ called a formal group algebra which depends on a commutative ring $R$, a lattice $\La$ lying between the weight lattice and root lattice of a finite root system $\Phi$, and a formal group law $\FGL$; then they defined a generalized version of the Demazure and BGG-operators (see \cite{Dem} and \cite{BGG} respectively) acting on $\RLF$, the formal Demazure/Push-pull operators. Let $W$ be the Weyl group of $\Phi$ and let $Q^F$ be the quotient field of $\RLF$. Following the ideas of \cite{KK} they introduced the so-called formal twisted algebra $Q_W^F$, which is the smash product of the group algebra $R[W]$ and of $Q^F$. Finally, they proved that a subring $D_F$ (resp. $\ADF$) of $Q^F_W$ generated by the Demazure/Push-pull elements (and multiplications) is isomorphic to the (affine) nil-Hecke algebra. For related results in the topological context we refer to the papers \cite{BE}, \cite{Co} \cite{GR}, \cite{HHH} by Bressler, Cooper, Evans, Ganter, Harada, Henriques, Holm, and Ram.

In the present paper we extend the construction of Calm\`{e}s et al.~to an arbitrary Kac-Moody root system and the hyperbolic formal group law
\[\FGL_{\mu_1,\mu_2}(u,v) = \tfrac{u+v-\mu_1 uv}{1+\mu_2 uv},\quad \mu_1,\mu_2\in R.\]
The hyperbolic formal group law is a natural choice since both the additive (corresponding to usual cohomology) and the multiplicative (corresponding to $K$-theory) formal group laws can be obtained from it by specialization. It has been actively studied in the context of elliptic formal group laws by Buchstaber-Bunkova \cite{BB}, \cite{BB1} and it has a rich topological background as it corresponds to the celebrated 2-parameter Todd genus introduced and studied by Hirzebruch in \cite{Hirz}. Recently in \cite{LZ} and \cite{LZ1}, it was used to generalize the root polynomial approach of Billey-Graham-Willems to Schubert calculus. 

This paper is structured as follows. In section~2, we recall the definitions of a commutative formal group law, formal group algebra, and facts concerning Kac-Moody root systems. In section~3, we introduce an analogue of the intermediate lattice $\La$ in the Kac-Moody setup, called the formal Demazure lattice, see Definition~\ref{dfn:extlattice}. In section~4, we show that all the definitions and properties of the formal (affine) Demazure operators and algebras hold for such a lattice, see \ref{relationsDF}. Finally, in section~5, we prove that for the hyperbolic formal group law the algebra $D_F$ is isomorphic (after extending the coefficients) to the Hecke algebra of the Weyl group of a Kac-Moody root system (see Theorem~\ref{thm:main}), thereby generalizing \cite[Prop.~9.2]{CZZ2}. \sm

{\it Acknowledgements.} I am grateful to my supervisors E.~Neher and K.~Zainoulline for their help and support and for introducing me to the subject. I would also like to thank Changlong Zhong for his useful comments. I am also very thankful to the referee for his or her comments and suggestions.

\section{Preliminaries}\label{sec:prel}

\subsection*{Hyperbolic formal group algebra}
A \emph{one-dimensional commutative formal group law} over a commutative unital ring $R$ is a power series $F(u,v) \in \Ruv$ such that (see \cite[p.1]{Haz})
\[
\FGL(\FGL(u,v),w) = \FGL(u,\FGL(v,w)), \quad
\FGL(u,v) = \FGL(v,u), \quad
\FGL(u,0) = u.
\]
The \emph{inverse} of $\FGL(u,v)$ is the unique power series $G(t) \in R\llbracket t \rrbracket$ such that $F(u,G(u)) = 0$ (see \cite[Appendix A.4.7]{Haz} for a proof). For simplicity we will write $u+_F v :=\FGL(u,v)$ and $-_F u := G(u)$. By the very definition of $F$ we have
\[
\FGL(u,v)=u+v +\sum_{i,j\geq 1} c_{ij}u^iv^j,\quad \text{where}\; c_{ij}\in R.
\]

Our central example is the following (see \cite[Example~63]{BB}, \cite[Corollary~3.8]{BB1})
\begin{ex} 
Let $\mu_1,\mu_2 \in R$. The \emph{hyperbolic 2-parameter formal group law} is defined as
\[
\FGL_{\mu_1,\mu_2}(u,v) = \tfrac{u+v-\mu_1 uv}{1+\mu_2 uv}= (u+v-\mu_1 uv)(\sum_{i\geq 0} (-\mu_2uv)^i).
\]
By \cite[p.~3,8]{BB}, its exponential is given by $\exp_{F_{\mu_1,\mu_2}}(u)=\tfrac{e^{\al u}-e^{\be u}}{\al e^{\al u} - \be e^{\be u}}$ where $\mu_1 =\al + \be$, $\mu_2= - \al \be$, hence the name hyperbolic formal group law. If $\mu_1=\mu_2=0$ (resp.\ $\mu_2=0, \mu_1 \in R^{\times}$) we get the \emph{additive formal group law} $F_a(u,v) = u+v$ (resp. the \emph{multiplicative periodic formal group law} $F_m(u,v) = u+v-\mu_1 uv$).

The inverse of $\FGL_{\mu_1,\mu_2}(u,v)$ is the same as for the multiplicative one, i.e.
\[
-_{F_{\mu_1,\mu_2}}u=- \sum_{n\geq 0} (-\mu_1)^nu^{n+1}.
\]

\end{ex}

Following \cite[Definition~2.4]{CPZ} let $\La$ be an abelian group and let $\FGL(u,v)$ be a one-dimensional formal group law. Let $R[x_{\la}]_{\la \in \La}$ be the ring of polynomials with indeterminates $x_{\la}$ for $\la \in \La$. Let $\eps \colon R[x_{\la}]_{\la \in \La} \to R$,  $x_{\la} \mapsto 0$ be the \emph{augmentation map}. Let $R\llbracket x_{\la}\rrbracket_{\la \in \La} := \Rhat_{\ker(\eps)}$ be the completion of $R[x_{\la}]_{\la \in \La}$ with respect to the ideal $\ker(\eps)$. Let $\JF$ be the closure in the $\ker(\eps)$-adic topology of the ideal of $R\llbracket x_{\la}\rrbracket_{\la \in \La}$  generated by $x_0$ and $x_{\la_1+\la_2} - (x_{\la_1} +_F x_{\la_2})$ for all $\la_1,\la_2 \in \La$. The quotient algebra
\begin{equation*}
\RLF = R\llbracket x_{\la}\rrbracket_{\la \in \La}/\JF
\end{equation*}
is called the \emph{formal group algebra}. If $F=F_{\mu_1,\mu_2}$ the corresponding formal group algebra will be called \emph{hyperbolic}. It is a unital associative commutative $R$-algebra.

By definition we have $x_{\la_1+\la_2} = x_{\la_1} +_F x_{\la_2}$ in $\RLF$. Observe that if $\La$ is free abelian of rank $n$, then $\RLF$ is  isomorphic (non-canonically) to the ring of formal power series $R\llbracket x_1,\ldots,x_n\rrbracket$ (we refer to \cite[\S2]{CPZ} for further properties and examples of formal group algebras). Note that we take the completion of $R[x_{\la}]_{\la \in \La}$ since in general the formal group law is a formal power series. 

\subsection*{Kac-Moody root systems}
Following \cite{K}, \cite{Ku}, \cite{MP}, and \cite{R}, let $I=\{1,...,l\}$ for some $l \in \NN$ and let $A=(a_{ij})_{i,j \in I}$ for $a_{ij} \in \ZZ$ be a \emph{generalized Cartan matrix}, i.e. $a_{ii}=2$, $a_{ij}\leq 0$ for $i\neq j$ and $a_{ij}=0$ implies $a_{ji}=0$. Choose a triple $(\lieh, \Pi, \Pi\ch )$ where $\lieh$ is a vector space over a base field $\kk$ ($\cha(\kk)=0$) of dimension $2l-\rk(A)$, $\Pi=\{\al_i \mid i\in I\} \subset \lieh^*$ and $\Pi\ch = \{\al_i\ch \mid i\in I \} \subset \lieh$ are linearly independent sets satisfying $\al_j(\al_i\ch)=a_{ij}$. This triple is unique up to a canonical isomorphism. The \emph{Kac-Moody algebra} $\lieg=\lieg(A)$, in the sense of \cite{Ku}, \cite{MP}, and  \cite{R}, is the Lie algebra over $\kk$, generated by $h \in \lieh$ and symbols $e_i$ and $f_i$ ($i\in I$) with the defining relations
\[
[\lieh,\lieh] = 0, \quad [h,e_i] =\al_i(h) e_i,\quad [h,f_i] = -\al_i(h) f_i,\quad [e_i,f_j] = \de_{ij} \al_i\ch,
\]
for all $h\in \lieh$ and $(\ad e_i)^{1-a_{ij}} (e_j)=0$, $(\ad f_i)^{1-a_{ij}} (f_j)=0$ for $i\neq j$. Note that this definition of a Kac-Moody algebra is equivalent to the definition of \cite{K} if the generalized Cartan matrix is symmetrizable.

Let $\La_r = \bigoplus_{i\in I} \ZZ \al_i \subset \lieh^*$ be the \emph{root lattice} (the root lattice is denoted by $Q$ in \cite{Ku}, here we follow the notation of \cite{T}). We have a root space decomposition
\[
\lieg = \lieh \oplus \sum_{\al \in \Phi} \lieg_{\al}
\]
where $\lieg_{\al} = \{ x \in \lieg \mid [h,x] = \al(h)x, \forall h\in \lieh \}$ and $\Phi =\{\al \in \La_r \setminus \{0\} \mid \lieg_{\al} \neq 0\}$. The set $\Phi$ is the \emph{Kac-Moody root system} corresponding to $\lieg$ with \emph{simple roots} $\al_i$ and \emph{simple coroots} $\al_i\ch$. 

For any $i\in I$, let $s_i \in \Aut(\lieh^*)$ be defined as $s_i(\phi) = \phi - \phi(\al_i\ch)\al_i$ for $\phi \in \lieh^*$. Let $W\subset \Aut(\lieh^*)$ be the subgroup generated by $\{s_i \mid i\in I\}$, called the Weyl group of $\lieg$. According to \cite[Prop.~3.13]{K} or \cite[Prop.~1.3.21]{Ku}, $W=\langle s_i\rangle_{1\leq i \leq l}$ is a Coxeter group where the order $m_{ij}$ of $s_is_j$ ($i,j \in I, i\neq j$) is given as follows:

\begin{center}
\begin{tabular}{|c||c|c|c|c|c|}
\hline
$a_{ij}a_{ji}$ & 0 & 1 & 2 & 3 & $\geq 4$ \\
\hline
$m_{ij}$ & 2 & 3 & 4 & 6 & $\infty$ \\
\hline
\end{tabular}
\end{center}

\sm

A root $\al \in \Phi$ is called \emph{real} if there exist $w\in W$ such that $\al=w(\al_i)$ for some simple root $\al_i$. We denote the set of real roots by $\Phireal$. For any $\al \in \Phireal$, we have $\al=w(\al_i)$ for some $w\in W$, $\al_i \in \Pi$ and we define $\al \ch := w(\al_i \ch)$. One can show that this is well-defined. Then for any $\be \in \Phi$ and $\al \in \Phi^{re}$ there exists $s_{\al} \in \GL(\lieh^*)$ such that $s_{\al}(\be) = \be - \langle \be, \al\ch \rangle \al$ and $s_{\al} \in W$ since $s_{\al}=ws_{\al_i}w^{-1}$.

\subsection*{Classification}
By the fundamental result of E.B.~Vinberg \emph{\cite{V}}, Kac-Moody root systems can be classified as follows. Let $A$ be an indecomposable generalized Cartan matrix. Then one and only one of the following three possibilities holds for $A$ \cite[p.~48]{K}:
\begin{description}
\item[(Fin)] $\det(A) \neq 0$; there exists $u>0$ such that $Au>0$; $Av \geq 0$ implies $v\geq 0$.
\item[(Aff)] $\cork(A)=1$; there exists $u>0$ such that $Au=0$; $Av \geq 0$ implies $Av=0$.
\item[(Ind)] there exists $u>0$ such that $Au<0$; $Av \geq 0, v\geq 0$ imply $v=0$.
\end{description}
We will say that a Kac-Moody root system is of \emph{finite, affine or indefinite type} if the corresponding generalized Cartan matrix satisfies (Fin), (Aff) or (Ind) respectively.

The Kac-Moody root systems of finite type correspond to the root systems of finite-dimensional semisimple Lie algebras since these are the corresponding Kac-Moody algebras. There is a complete classification of Kac-Moody root systems of affine type in terms of affine Dynkin diagrams (see \cite[p.~54-55]{K}).

For a Kac-Moody root system of affine type, it is traditional to start numbering the simple roots at 0, i.e. the root basis is $\Pi=\{\al_i \mid i\in I\}$ with $I=\{0,...,l\}$. In this case, there exists a unique element $\de \in \La_r$ called \emph{null root} (which will play an important
role in our arguments) defined as $\de = \sum_{i=0}^l a_i \al_i$ where the $a_i$'s are the numerical labels of each node in the Dynkin diagram of \cite[p.~54-55]{K}. Note that $\al_0$ is not arbitrary since $\al_0=a_0^{-1}(\delta-\theta)$ where $\theta$ is the unique highest root of the underlying root system of finite type. By definition, for $i \in I$ we have \[\langle \de, \al_i\ch \rangle =0.\] Moreover, there exist an element $d^* \in \lieh^*$ (defined up to a summand proportional to $\de$) such that \begin{equation}\label{def:d} \langle d^*,\al_0\ch \rangle = 1, \quad \text{ and } \langle d^*, \al_j\ch \rangle =0\end{equation} for $0 \leq i \leq l, 1 \leq j \leq l$.

\section{Formal Demazure lattices}

Note that in the case of a finite root system with root lattice $\La_r$ and root basis $\Pi = \{ \al_i \mid i \in I\}$ the weight lattice 
\[
\La_w = \{\la \in \lieh^* \mid \langle \la, \La_r \ch \rangle \subset \ZZ \} = \{\la \in \frh^* \mid \lan \la , \al_i\ch\ran \in \ZZ \text{ for all } i\in I\}
\]  
plays an important role in the classification and representation theory of the respective algebras/groups up to isogeny. Namely, each isogeny class of algebraic groups corresponds to an intermediate lattice $\La$, $\La_r \subseteq \La \subseteq \La_w$ with $\La=\La_r$ (resp.\ $\La=\La_w$) being the adjoint (resp.\ simply-connected) class. Moreover, these intermediate lattices are the initial data in the construction of the equivariant oriented cohomology ring of flag varieties of \cite{HMSZ} and \cite{CZZ}. As the main purpose of the present paper is to extend these results to the Kac-Moody case, a natural question arises of what should be taken as an analogue of the intermediate lattice $\La$ or, equivalently, to which lattices $\La\supseteq \La_r$ in the Kac-Moody case can one extend the calculus of formal Demazure and BGG operators of \cite{HMSZ} and \cite{CZZ}.

Let now $\Phi$ be a Kac-Moody root system and let $\Pi= \{ \al_i \mid i\in I\}$ be a set of simple roots. We use the notation introduced in \S\ref{sec:prel}. We call $\{ \om_i \mid i \in I\} \subset \frh^*$ a {\em set of fundamental weights} if $\lan \om_i,  \al_j \ch \ran = \de_{ij}$. Since in general the $\al_i\ch$ do not span $\frh$, a set of fundamental weights is not uniquely determined by the data defining $\Phi$. For example, if $\Phi$ is of affine type and $\{\om_i \mid i \in I\}$ is a set of fundamental weights, then so is $\{ \om_i + n \de \mid i \in I\}$ for any $n\in \kk$ since  $\langle \de,\al_j\ch \rangle = 0$. This non-uniqueness is the reason why the {\em standard weight lattice} $\La_w$ may be too large for our purposes. The following notion provides a candidate for a replacement of $\La_w$.

\begin{dfn}[Formal Demazure lattices]\label{dfn:extlattice}
We call a finitely generated free subgroup $\La$ of $(\lieh^*, +)$ a \emph{formal Demazure lattice} if it has the following properties:
\begin{description}

\item[(FDL1)] every simple root can be extended to a basis of $\La$, in particular $\La_r \subset \La$,

\item[(FDL2)] $\langle \La, \al_i\ch \rangle \subset \ZZ, \forall \al_i \in \Pi$, i.e., $\La \subset \La_w$.
\end{description}

Let $\La$ and $\La'$ be formal Demazure lattices. By definition, a {\em morphism} of formal Demazure lattices is a  $\ZZ$-module homomorphism $\phi\colon \La \to \La'$ satisfying $\phi(\La_r) \subset
\La_r$. Together with this notion of a morphism, formal Demazure lattices form a category. In particular, we have the notion of an isomorphism of formal Demazure lattices.
\end{dfn}

\begin{rems}\label{rem1}
(a) A formal Demazure lattice $\Lambda$ is not necessarily a lattice of $\lieh^*$ in the usual sense since a basis of $\Lambda$ need not generate $\lieh^*$. This small abuse of language is traditional.

(b) By \cite[VII, \S4.2, Lemme~1]{Bo}, cf.\ \cite[Lemma~12.7]{CZZ}, the condition (FDL1) is equivalent to either one of the conditions (FDL1') or (FDL1''):

\begin{description}
  \item[(FDL1')] The coordinates of every simple root with respect to some
      $\ZZ$-basis of $\La$ form a unimodular vector.
  \item[(FDL1'')] The coordinates of every simple root with respect to
      every $\ZZ$-basis of $\La$ form a unimodular vector.
\end{description}

(c) Obviously, the root lattice $\La_r = \bigoplus_i \ZZ \al_i$ is a so-called {\em trivial} formal Demazure lattice. We are interested in bigger formal Demazure lattices. Examples will be given in \ref{ex:E} and \ref{ex:A}.
\end{rems}

\begin{lem}
A formal Demazure lattice has the following properties.
\begin{enumerate}
\item[(i)] $w(\La) = \La$ for all $w \in W$,
\item[(ii)] $\langle \La, \al \ch \rangle \subset \ZZ$ for any real root $\al \in \Phireal$,
\item[(iii)] every real root can be extended to a basis of $\La$.
\end{enumerate} \label{formal Demazure:real:roots}
\end{lem}

\begin{proof}
(i) Let $\la \in \La$. Let $s_{\al_i} \in W$ be a simple reflection. Then $s_{\al_i}(\la)= \la - \langle \la, \al_i \ch \rangle \al_i = \la - c\al_i$ for some $c \in \ZZ$ by (FDL2). Since by (FDL1) we have $\La_r \subset \La$, $\al_i \in \La$ so $c\al_i \in \La$ and $s_{\al_i}(\la) \in \La$. Since $W$ is generated by simple reflections and $s_{\al_i}^2=\id$, we have $w(\La) = \La$ for any $w \in W$.

(ii) Let $\la \in \La$ and let $\al \in \Phireal$. Then there exist $\al_i \in \Pi$ and $w \in W$ such that $w(\al_i)=\al$. Therefore, by (FDL2), $\langle \la, \al \ch \rangle = \langle \la, (w\al_i))\ch \rangle = \langle w^{-1}(\la), \al_i\ch \rangle \in \ZZ$ by (i).

(iii) Let $\al \in \Phireal$. As before, $\al = w(\al_i)$ for some $w\in W$ and some simple root $\al_i$. By (FDL1), $\al_i$ can be completed to a basis $B$ of $\La$. Then $w(B)$ is a basis of $\La$ containing $\al$ and $w(B) \subset \Lambda$ by (a).
\end{proof}

\begin{ex}[$\Phi = \rma_1^{(1)}$]\label{ex:E}
Let $\Phi$ be the affine root system $\rma^{(1)}_1$, associated with the generalized Cartan matrix $A = \begin{tiny} \begin{pmatrix} 2 & -2 \\ -2 & 2 \end{pmatrix} \end{tiny}$ of affine type. We will describe all formal Demazure lattices which are lattices (in the usual sense) in $\QQ \al_0 \oplus \QQ \al_1$.

Any lattice $\La$ of $\QQ \al_0 \oplus \QQ \al_1$ has the form $\ZZ \la_0 \oplus \ZZ \la_1$ for $\QQ$-linearly independent $\la_i$. Hence there exists an invertible matrix
\[ 
B= \begin{pmatrix}
  a & b \\ c & d\end{pmatrix} \in \GL_2(\QQ)
\]
such that $\la_0 = a \al_0 + c \al_1$ and $\la_1 = b \al_0 + d \al_1$. For such a lattice $\La$ the condition $\al_i \in \La$, which is part of (FDL1), holds if and only if $B^{-1}\in {\rm Mat}_2(\ZZ)$, i.e.
\begin{equation}\label{ex:E1}
    B = \textstyle \frac{1}{\De'} B', \quad \hbox{where }
     B' = \begin{pmatrix}
      a' & b' \\ c' & d' \end{pmatrix} \in {\rm Mat}_2(\ZZ) \hbox{ and }
          \De' = \det(B').
\end{equation}

Since the coordinates of $\al_0$ and $\al_1$ with respect to the basis $\{\la_0, \la_1\}$ of $\La$ are $(d'\ -c')$ and $(-b'\ a')$ respectively, it follows that
\begin{equation}\label{ex:E3}
 \hbox{(FDL1)} \qquad \iff \qquad \hbox{$(a'\ b')$ and $(c'\ d')$ are unimodular.}
\end{equation}
Moreover we have
\[
   \lan \la_0, \al_0\ch\ran = a \lan \al_0,
\al_0\ch\ran + c \lan \al_1, \al_0\ch\ran = 2 (a-c) =\textstyle  \frac{2}{\De'} (a'-c')
   = - \lan \la_0, \al_1\ran, 
 \]
And similarly, $\lan \la_1, \al_0\ch\ran = 2 (b-d) = \frac{2}{\De'}(b'-d') = - \lan \la_1, \al_1\ch\ran$. Thus
\begin{equation}\label{ex:E2}
 \hbox{(FDL2)} \quad \iff \quad a'-c'\in \textstyle
\frac{\De'}{2} \ZZ
   \quad \hbox{and} \quad d'-b' \in \frac{\De'}{2} \ZZ.
\end{equation}
For example, a matrix $B'$ with $c'=0$ satisfies the conditions \eqref{ex:E2} and \eqref{ex:E3} if and only if there exists $m\in \ZZ$ such that
\[
   B' = \begin{pmatrix} a' & \eps(1 - \frac{a'}{2}m) \\ 0 & \eps   \end{pmatrix},
\quad \eps = \pm 1, \quad \hbox{and} \quad
    (a', \ 1-\textstyle\frac{a'}{2} m) \hbox{ is unimodular.}
\]
Equivalently, the first row of $B'$ is either $(n,\epsilon(1-nm)) $ or $(4n, \epsilon(1-2nm))$ for some integers $n$ and $m$. In particular, for every $n \in \ZZ\setminus \{0\}$ the matrix
\[
    B_n = \tfrac{1}{4n} \begin{pmatrix} 4n & 1+2n \\ 0 & 1 \end{pmatrix}
\]
satisfies the conditions \eqref{ex:E3} and \eqref{ex:E2} and thus defines a formal Demazure lattice $\La_n$, namely
\begin{equation} \label{ex:E4}
 \La_n = \ZZ \al_0 \oplus \ZZ
((\tfrac{1}{4n}+\tfrac{1}{2})\al_0 + \tfrac{1}{4n} \al_1)
\end{equation}
\end{ex}

We will study these formal Demazure lattices in the following lemma.

\begin{lem} \label{lem:ex:e}
Let $n,m \in \ZZ \setminus \{0\}$ and let $\Lambda_n$ and $\Lambda_m$ be defined by \emph{(\ref{ex:E4})}.

{\rm (a)} Then $\La_m \subset \La_n$ if and only if $\tfrac{n}{m}$ is an odd integer. Hence we have an infinite chain of formal Demazure lattices
\[
   \La_1 \subsetneq \La_3 \subsetneq \La_9 \subsetneq \cdots
\]

{\rm (b)} $\La_n \cong \La_m$ if and only if $n=m$.
\end{lem}

\begin{proof} 
(a) We have $(\tfrac{1}{4m}+\tfrac{1}{2})\al_0 + \tfrac{1}{4m} \al_1 \in \La_n $ $\Leftrightarrow$ $\exists a,b \in \ZZ$ such that  $a\al_0 + b((\tfrac{1}{4n}+\tfrac{1}{2})\al_0 + \tfrac{1}{4n} \al_1)) = (\tfrac{1}{4m}+\tfrac{1}{2})\al_0 + \tfrac{1}{4m} \al_1$ $\Leftrightarrow$ $a+b(\tfrac{1}{4n})=\tfrac{1}{4m}+\tfrac{1}{2} \text{ and } b(\tfrac{1}{4n})=\tfrac{1}{4m}$  $\Leftrightarrow$ $\tfrac{n}{m} \in \ZZ \text{ and } 2\mid(1-\tfrac{n}{m})$ which implies the result.

(b) Since $(\tfrac{1}{4n}+\tfrac{1}{2})\al_0 + \tfrac{1}{4n} \al_1$ is of order $4n$ in $\La/\La_r$, we have $\La_n/\La_r \cong \ZZ_{4n}$, hence $\La_n \cong \La_m$ if and only if $n=m$.
\end{proof}

\begin{exs}\label{ex:A} 
(i) We consider again the root system $\Phi$ of Example~\ref{ex:E}. If $\La \subset \kk \al_0 \oplus \kk \al_1$ is a formal Demazure lattice,  then so is $\ZZ d^* \oplus \La$. Indeed, since $\frh^* = \kk d \oplus \kk \al_0 \oplus \kk \al_1$, this follows immediately from the definition and \eqref{def:d}. \sm

(ii) Let $\Phi$ be an affine root system. Recall the null root $\de$ defined in section 2.2. If
\begin{equation}
  \label{ex:A1} \hbox{there exists a pair $i,j$, $i\neq j$ such that $a_i=a_j=1$},
\end{equation}
then for any $m\in \ZZ \setminus \{0\}$
\[
\La = \ZZ
d^* + \ZZ \al_0 + \cdots +\ZZ\al_{i-1} +  \ZZ \tfrac{1}{m} \de +
    \ZZ \al_{i+1} + \cdots + \ZZ \al_l
\]
is a formal Demazure lattice since $$ \al_i = 0\cdot d^* + (-a_0) \al_0 + \cdots + (-a_l) \al_l + m (\tfrac{1}{m} \de)$$ and $(0,-a_0,\ldots , -a_l, m)$ is unimodular. The condition \eqref{ex:A1} is fulfilled for $\Phi$ of type $A^{(1)}_l$, $B^{(1)}_l$,  $C^{(1)}_l$, $D^{(1)}_l$, $E^{(1)}_6$, $E^{(1)}_7$, $A^{(2)}_{2l-1}$, $D^{(2)}_{l+1}$, $E^{(2)}_6$ and $D^{(3)}_4$. 

For case $G^{(1)}_2$, $\La = \ZZ \tfrac{1}{m} \de + \ZZ \al_1 + \ZZ \al_2$ is a formal Demazure lattice for any $m\in \ZZ$ since $\de= \al_0 +2 \al_1 + 3 \al_2$ and $(m,-2,-3)$ is unimodular. 

Similarly, for case $F^{(1)}_4$ and $E^{(1)}_8$, $\La = \ZZ \tfrac{1}{m} \de + \ZZ \al_1 + \ZZ \al_2 + \cdots$ is a formal Demazure lattice since $a_1=2, a_2=3$ and $(m,-2,-3,\ldots )$ is unimodular. 

For the two remaining cases, $A^{(2)}_2$ and $A^{(2)}_{2l}$, $\La = \ZZ \al_0 + \cdots + \ZZ \tfrac{1}{m} \delta$ is a formal Demazure lattice if $m$ is an odd integer. 

(iii) Let $\Phi$ be an arbitrary Kac-Moody root system and let $\{\omega_i \mid i\in I\}$ be a set of fundamental weights. Then $\La= \sum_{i\in I} \ZZ \omega_i$ is a formal Demazure lattice unless $\al_i (\al_j\ch) \in 2\ZZ$ for some $j$ and for all $i\in I$ which happens if and only if we have a column of even integers in the generalized Cartan matrix corresponding to $\Phi$. It is a formal Demazure lattice since $\al_i = \sum_{j\in I} \al_i(\al_j\ch) \omega_j$ and $( \al_i (\al_j\ch))_{j\in J}$ is a unimodular vector so that (FDL1') holds. \sm

(iv) We can also have formal Demazure lattices for root systems of indefinite type. For example, let $A = \begin{tiny} \begin{pmatrix} 2 & -4 \\ -4 & 2 \end{pmatrix} \end{tiny}$. Then $\La = \ZZ \al_1 + \ZZ (\tfrac{1}{2} \al_1 + \tfrac{1}{2} \al_2)$ is a formal Demazure lattice.
\sm

(v) If $\Phi$ is a finite root system, i.e., the generalized Cartan matrix is a Cartan matrix, any abelian group $\La$ satisfying $\La_r \subset \La \subset \La_w$ is free of finite rank and satisfies (FDL2). It is shown in \cite[Lemma~2.1]{CZZ} that any such $\La$ satisfies (FDL1) and hence is a formal Demazure lattice, except when $\La=\La_w$ and $\Phi$ is of type $\rmc_l$, $l \ge 1$. \sm

(vi) Let $\Phi$ be a Kac-Moody root system whose associated generalized Cartan matrix $A$ is invertible. Then $\La_w/\La_r$ is finite, namely $|\La_w/ \La_r| = | \det(A)|$. Hence there are only finitely many formal Demazure lattices in this case. 
\end{exs}

\begin{rems}
(1) In this remark, let $k=\CC$. Following \cite[Section~6.1.6]{Ku}, we define an \emph{integral Cartan subalgebra} $\lieh_{\ZZ}$ of $\lieg$ as a finitely generated $\ZZ$-submodule of $\lieh$ satisfying the following conditions:
\begin{description}
\item[(C1)] $\lieh_{\ZZ}$ is an integral form of $\lieh$, i.e., the natural map $\lieh_{\ZZ} \otimes_{\ZZ} \CC \to \lieh$ is an isomorphism,
\item[(C2)] all simple coroots $\al_i \ch \in \lieh_{\ZZ}$,
\item[(C3)] $\lieh^*_{\ZZ} := \Hom_{\ZZ}(\lieh_{\ZZ},\ZZ) \subset \lieh^*$ contains all simple roots $\al_i$, and
\item[(C4)] $\lieh_{\ZZ}/\sum_{i=1}^n \ZZ \al_i \ch$ is torsion free.
\end{description}
However, for an integral Cartan subalgebra $\lieh_{\ZZ}$, $\lieh^*_{\ZZ}$ is not necessarily a formal Demazure lattice. For example, if $\Phi=A^{(1)}_1$ and $\lieh_{\ZZ}= \ZZ d \oplus \ZZ \al_0 \ch \oplus \ZZ \al_1 \ch$, then since $\al_1(d)=0, \al_1(\al_0 \ch)=-2$, and $\al_1(\al_1 \ch)=2$ we have that $\tfrac{1}{2} \al_1 \in \lieh^*_{\ZZ}$, hence $\al_1$ cannot be extended to a basis of $\lieh^*_{\ZZ}$. This situation occurs only if the Cartan matrix of $\Phi$ has a column of even integers. \sm

\noindent
(2) Following \cite[Section~6.1]{MP}, we can define a \emph{restricted weight lattice} as a subgroup $\La$ of $\lieh^*$ satisfying the following properties:
\begin{description}
\item[(RWL1)] $\La \subset \La_w$,
\item[(RWL2)] $\langle \La, \al_i \ch \rangle \subset \ZZ, \forall i \in I$,
\item[(RWL3)] there exists a minimal regular weight $\rho \in \La$, i.e. $\langle \rho, \al_i\ch \rangle = 1, \forall i\in I$,
\item[(RWL4)] $\La$ is a free $\ZZ$-module with a basis $B$ consisting of $\kk$-linearly independent elements and such that $B$ contains a set of fundamental weights.
\end{description}
Again, we have that a restricted weight lattice is not necessarily a formal Demazure lattice. For example, if $\Phi=A^{(1)}_1$, $\La= \ZZ d^* \oplus \ZZ \tfrac{1}{2} \al_1 \oplus \ZZ \de$ is a restricted weight lattice because $\{\omega_0=d^*, \omega_1=d^* + \frac{1}{2} \al_1 \}$ is a set of fundamental weights. However, it is not a formal Demazure lattice since $\tfrac{1}{2}\al_1$ cannot be extended to a basis of $\La$. Also, any set of fundamental weights is of the form $\{\omega_0 + m\de, \omega_1+n \de\}$ for some $m,n\in \ZZ$. So we have that $\La = \ZZ d^* \oplus \ZZ \al_0 \oplus \frac{1}{2} \de$ is a formal Demazure lattice which is not a restricted weight lattice since $\omega_1 + n \de \notin \La$ for any $n \in \ZZ$.
\end{rems}

\section{Formal affine Demazure algebras}

The purpose of the present section is to extend the definitions of the formal Demazure operator of \cite{CPZ} and of the formal affine Demazure algebra of \cite{HMSZ} and \cite{CZZ} to formal Demazure lattices. We follow closely \cite[\S3]{CPZ} and \cite[\S2]{CZZ}.

In this section $\La$ is a formal Demazure lattice.

\begin{dfn}
For $\al \in \Lambda$, we will say that an element $x_{\al}$ of $\RLF$ is \emph{regular} if $x_{\al}$ is not a zero divisor of $\RLF$.
\end{dfn}

\begin{lem}[{cf.~\cite[Lem.~2.2]{CZZ}}]
For $\al \in \Phireal$, $x_{\al}$ is regular in $\RLF$. \label{regular}
\end{lem}

\begin{proof}
Since $\Lambda$ is a formal Demazure lattice and $\al \in \Phireal$ is a real root, $\al$ can be completed to a basis of $\La$ by \ref{formal Demazure:real:roots}. Also, by \cite[Corollary 2.12]{CPZ} we have $\RLF \simeq R\llbracket x_1,\ldots ,x_n\rrbracket $ with $x_1 = x_{\al}$. Therefore, $x_{\al}$ is regular by \cite[Lemma~12.3]{CPZ}. 
\end{proof}

\begin{rem}
Let $\La=\lieh^*_{\ZZ}$ where $\lieh_{\ZZ}$ is any integral Cartan subalgebra of $\lieg$. Suppose the Cartan matrix of $\Phi$ has no column of even integers. Then by an argument similar to the proof of \cite[Lemma~2.2]{CZZ}, $x_{\al}$ is regular in $\RLF$ for any real root $\al \in \Phireal$. If the Cartan matrix of $\Phi$ has a column of even integers, $x_{\al}$ is regular if 2 is regular in $R$. The same situation occurs if $\La$ is a restricted weight lattice.
\end{rem}

Following \cite[Corollary~3.4]{CPZ}, we can show that $u-s_{\al}(u)$ is uniquely divisible by $x_{\al}$ for any $u \in \RLF$ if and only if $x_{\al}$ is regular. Since $x_{\al}$ is regular for any $\al \in \Phireal$, we can define an operator on $\RLF$ as follows.
 
\begin{dfn}[{\cite[Def.~3.5]{CPZ}}]
Let $\al \in \Phireal$. We define a $R$-linear operator on $\RLF$ as
\begin{equation*}
\Delta_{\al}(u) = \tfrac{u-s_{\al}(u)}{x_{\al}}
\end{equation*}
\noindent for all $u\in \RL$. We call the $\Delta_{\al}$'s \emph{Demazure operators}.
\end{dfn}

We now follow \cite[\S6]{HMSZ} and \cite[\S5]{CZZ}. Let $Q^F$ be the localization of $\RLF$ at the multiplicative subset generated by $\{ {x_{\al}} \mid \al \in \Phireal \}$. The localization map is injective (see \cite[Lemma~3.2]{CZZ}) because $x_{\al}$ is regular for all $\al \in \Phireal$ by \ref{regular}. Since $W$ preserves the set of real
roots, the action of $W$ on $\RLF$ extends to an action on $Q^F$.

\begin{dfn}
\cite[Def.~6.1]{HMSZ} Let $\de_w$ be the element in the group algebra $R[W]$ corresponding to $w\in W$ and $1 := \de_1 \in R[W]$. We define the \emph{twisted formal group algebra} to be the $R$-module $Q_W^F := R[W] \otimes_R Q^F$. We abbreviate $xq := x \otimes q$ for $x\in R[W], q\in Q^F$ and define an $R$-bilinear multiplication by
\begin{equation*}
(\de_{w'} \psi')(\de_w \psi) = \de_{w'w} w^{-1}(\psi') \psi, \forall w,w' \in W, \psi,\psi' \in Q^F.
\end{equation*}
\end{dfn}

The algebra $Q_W^F$ is unital associative and the $W$-invariant elements lie in the center.

\begin{dfn}
\cite[Def.~6.2]{HMSZ} For each $\al \in \Phireal$ we define the \emph{formal Demazure element} as
\begin{equation*}
X_{\al} := \tfrac{1}{x_{\al}}(1-\de_{s_{\al}}) = \tfrac{1}{x_{\al}} - \tfrac{\de_{s_{\al}}}{x_{-\al}} \in Q_W^F.
\end{equation*}
\end{dfn}

Let $\Delta_i= \Delta_{\al_i}$ and $X_{i} = X_{\al_i}$. For a reduced expression $w=s_{\al_{i_1}}s_{\al_{i_2}} \cdots s_{\al_{i_n}}$ in $W$, we denote \[X_w= X_{\al_{i_1}}X_{\al_{i_2}} \cdots X_{\al_{i_n}} = X_{i_1} X_{i_2} \cdots X_{i_n}.\] We will also sometimes denote $X_iX_jX_i\ldots $ as $X_{iji\ldots }$. 

One can check that the following lemma holds in our setting.

\begin{lem}
\emph{\cite[Cor.~5.6]{CZZ}} The elements $(X_w)_{w\in W}$ form a basis of $Q^F_W$ as a left $Q^F$-module.
\end{lem}

\begin{dfn}
\cite[Def.~6.3]{HMSZ} The \emph{formal Demazure algebra} $\DF$ is the $R$-subalgebra of $Q^F_W$ generated by the formal Demazure elements $X_{\al}$ for all $\al \in \Phireal$. The \emph{formal affine Demazure algebra} $\ADF$ is the $R$-subalgebra of $Q^F_W$ generated by the elements of the formal group algebra $\RLF$ and by $\DF$. In fact, by \cite[Lemma 5.8]{CZZ}, $\ADF$ is generated by $X_{\al_i}$ for all simple roots $\al_i$ together with the elements of $\RLF$.
\end{dfn}

\begin{dfn}
\cite[Def.~4.2]{HMSZ} Consider the power series $g^F(u,v)$ defined by $u+_Fv = u+v-uvg^F(u,v)$ and, for $\al \in \Phi^{re}$, let 
\[
\kappa_{\al} := g^F(x_{\al},x_{-\al}) = \tfrac{1}{x_{\al}} + \tfrac{1}{x_{-\al}}.
\]
We have $\kappa_{\al} \in \RLF$ since $g^F(x_{\al},x_{-\al}) \in \RLF$ by definition of a formal group law. For all $\la, \mu \in \La \setminus \{0\}$ with $\la+\mu \neq 0$, we define 
\[
\kappa_{\la,\mu} := \tfrac{1}{x_{\la+\mu}} \left( \tfrac{1}{x_{\mu}} - \tfrac{1}{x_{-\la}}\right) - \tfrac{1}{x_{\la}x_{\mu}} \in Q^F.
\] 
In fact, $\kappa_{\la,\mu} \in \RLF$ by \cite[Lemma 6.7]{HMSZ}. We denote $\kappa_{a\al_i+b\al_j, c\al_i+d\al_j}$ as $\kappa_{ai+bj,ci+dj}$ and $x_{a\al_i+b\al_j}$ as $x_{ai+bj}$.
\end{dfn}

\begin{prop}[{\cite[Prop.~6.8]{HMSZ}}] 
Suppose $i,j\in I$ and let $m_{ij}$ be the order of $s_is_j$ in $W$. Then $$ \underbrace{X_j X_i X_j \cdots}_{m_{ij} terms} - \underbrace{X_i X_j X_i \cdots}_{m_{ij} terms} = \sum_{w\in W, 1 \leq l(w) \leq m_{ij}-2} X_w \eta_w$$ for some $\eta_w \in Q^F$. More precisely, we have the following braid relations:
\begin{itemize}
\item[(a)] If $\langle \al_i,\al_j \ch \rangle = 0$, so that $m_{ij}=2$, then $X_i X_j = X_j X_i$.
\item[(b)] If $\langle \al_i,\al_j \ch \rangle = \langle \al_j,\al_i \ch \rangle =-1$, so that $m_{ij}=3$, then $$X_{jij} - X_{iji} =  X_i \kappa_{i,j} - X_j \kappa_{j,i}.$$
\item[(c)] If $\langle \al_i,\al_j \ch \rangle = -1$ and $\langle \al_j,\al_i \ch \rangle = -2$, so that $m_{ij}=4$, then
\begin{align*} 
X_{jiji} - X_{ijij} &= X_{ij}(\kappa_{i+2j,-j} + \kappa_{j,i}) - X_{ji}(\kappa_{i+j,j} + \kappa_{i,j}) \\ 
&\quad + X_j(\Delta_i(\kappa_{i+j,j} + \kappa_{i,j})) - X_i(\Delta_j(\kappa_{i+2j,-j} + \kappa_{j,i})). 
\end{align*}
\item[(d)] If $\langle \al_i,\al_j \ch \rangle = -1$ and $\langle \al_j,\al_i \ch \rangle =-3$, so that $m_{ij}=6$, then
\begin{align*}
X_{jijiji} - X_{ijijij} &= X_{ijij}(\kappa_{j,i} + \kappa_{2i+3j,-i-2j}+\kappa_{-i-3j,i+2j} + \kappa_{i+2j,-j} ) \\ 
 &\quad \quad - X_{jiji}(\kappa_{i,j} + \kappa_{-2i-3j,i+2j}+\kappa_{-i-2j,i+3j} + \kappa_{i+j,j}) \\ 
 &\quad + X_{jij}(\Delta_i(\kappa_{i,j} + \kappa_{-2i-3j,i+2j}+\kappa_{-i-2j,i+3j} + \kappa_{i+j,j})) \\
 &\quad \quad - X_{iji}(\Delta_j(\kappa_{j,i} + \kappa_{2i+3j,-i-2j}+\kappa_{-i-3j,i+2j} + \kappa_{i+2j,-j}))\\ 
 &\quad + X_{ij} \xi_{ij} - X_{ji} \xi_{ji} + X_j(\Delta_i(\xi_{ji})) - X_i(\Delta_j(\xi_{ij})),
\end{align*}

\noindent
where $\xi_{ij} = \frac{1}{x_ix_{i+j}x_{i+2j}x_{2i+3j}} + \frac{1}{x_ix_jx_{i+2j}x_{-2i-3j}} + \frac{1}{x_ix_jx_{2i+3j}x_{-i-j}} \\ - \frac{1}{x_ix_{i+j}x_{i+2j}x_{-1-3j}} - \frac{1}{x_ix_{i+j}x_{i+3j}x_{-j}} + \frac{1}{x_{i+j}x_{i+3j}x_{-j}x_{-2i-3j}} +\frac{1}{x_{i+3j}x_{2i+3j}x_{-j}x_{-i-2j}} \\ + \frac{1}{x_{i+j}x_{i+2j}x_{-i-3j}x_{-2i-3j}} - \frac{1}{x_ix_jx_{i+2j}x_{i+3j}}$\\

\noindent and $\xi_{ji} = \frac{1}{x_ix_jx_{2i+3j}x_{-i-2j}} +
\frac{1}{x_ix_jx_{i+2j}x_{-i-3j}} + \frac{1}{x_jx_{i+2j}x_{i+3j}x_{2i+3j}}
\\- \frac{1}{x_ix_jx_{i+j}x_{2i+3j}} +
\frac{1}{x_{i+j}x_{i+2j}x_{-i}x_{-2i-3j}} +
\frac{1}{x_{i+3j}x_{2i+3j}x_{-i-j}x_{-i-2j}} +
\frac{1}{x_{i+j}x_{i+3j}x_{-i}x_{-i-2j}} \\
-\frac{1}{x_jx_{i+3j}x_{2i+3j}x_{-i-j}} -
\frac{1}{x_jx_{i+j}x_{i+3j}x_{-i}}$.
\end{itemize}
\label{propbraidrel}
\end{prop}

\begin{rem}
In the cases $(a),(b)$ and $(c)$, we have $\eta_w \in \RLF$ since $\kappa_{\la,\mu} \in \RLF$ for all $\la,\mu \in \Phi,\la+\mu \neq 0$ by \cite[Lemma~6.7]{HMSZ}. For case $(d)$, we can view $\Span_{\ZZ}\{\al_i,\al_j\}$ as a finite subroot system of $\Phi$ of rank 2 and we also have $\eta_w \in \RLF$ by
\cite[Lemma~7.1]{CZZ} and \cite[Example~7.3]{CZZ}. \label{remarkcoeff}
\end{rem}

\begin{exs} 
For $F(u,v)=F_a(u,v)$ or $F(u,v)=F_m(u,v)$, since $\kappa_{i,j}= \tfrac{x_i(x_{-i}-x_j)-x_{i+j}x_{-i}}{x_ix_{-i}x_jx_{i+j}}$ we have that $\kappa_{i,j} = 0$ hence the braid relations $X_{ji\ldots } = X_{ij\ldots }$ hold. For $F(u,v)=F_{\mu_1,\mu_2}(u,v)$, we get by direct computation that
\begin{itemize}
\item[(i)] $\kappa_{i,j} = \mu_2$ (see also \cite[Rem.~3.10 and Ex.~3.12.(3)]{Co}), so for $m_{ij}=3$ we have \[X_{jij} - X_{iji} =
    \mu_2(X_i - X_j),\] and for $m_{ij}=4$ we have \[X_{jiji} - X_{ijij} =
    2\mu_2(X_{ij} - X_{ji}),\]
\item[(ii)] $\xi_{ji} = \xi_{ij} = 3\mu_2^2$ (see computer aided computations in \cite[Appendix]{Le}), so for $m_{ij}=6$ we have
    \[X_{jijiji}-X_{ijijij} = 4\mu_2(X_{ijij} - X_{jiji}) + 3\mu_2^2(X_{ij}
    - X_{ji}).\]
\item[(iii)] $\kappa_i = \mu_1$, so $X_i^2=\mu_1 X_i$ for all $i\in I$.
\end{itemize} \label{specialelliptickappas}
\end{exs}

\begin{prop}[{cf.~\cite[Thm.~6.14]{HMSZ}}] 
Let $\La$ be a formal Demazure lattice. Let $F(u,v)=F_{\mu_1,\mu_2}(u,v)$ be the hyperbolic 2-parameter formal group law.

The formal affine Demazure algebra \emph{$\ADF$} is generated as an $R$-algebra by elements of $\RLF$ and the formal Demazure elements $X_i$, $i\in I$, and satisfies the following relations:
\begin{itemize}
\item $\gamma X_{\al} = X_{\al} s_{\al}(\gamma) + \Delta_{\al}(\gamma)$ for all
    $\al \in \Phireal$ and $\gamma \in \RLF$;
\item $X^2_i = \mu_1 X_i$ for all $i \in I$;
\item $X_iX_j = X_jX_i$ for all $i,j \in I$ if $m_{ij}=2$;
\item the braid relations (i) and (ii) of \ref{specialelliptickappas} for
    all $i,j \in I$ if $m_{ij} =3,4,6$;
\item no relations between $X_i$ and $X_j$ for all $i,j \in I$ such that
    $m_{ij}=\infty$.
\end{itemize}
These relations form a complete set of relations for \emph{$\ADF$}.
\label{relationsDF}
\end{prop}

\begin{proof}
This follows by the same proof as in \cite[Thm.~6.14]{HMSZ} and by the fact that both $\xi_{ji}$ and $\xi_{ij}$ lie in $\RLF$ by \ref{remarkcoeff} or in view of the explicit formulas~\ref{specialelliptickappas}.(ii).
\end{proof}

\begin{rem}
One can show that if we remove the first relation of \ref{relationsDF} we obtain a complete set of relations for the formal Demazure algebra $\DF$.
\end{rem}

\section{Hecke Algebras}

The purpose of the present section is to generalize \cite[Prop.~9.2]{CZZ1} to a Kac-Moody root system of arbitrary type, see \ref{thm:main}. In \ref{cor1}, we obtain an algebra isomorphic to the affine Demazure algebra.

\begin{dfn}
(cf.~\cite[~Def.~8.1]{HMSZ}) Let $R$ be a (commutative unital) ring containing $\ZZ[t,t^{-1}]$. In Lusztig's notation, the \emph{Hecke algebra} $H$ associated with the Coxeter group $W$ is the associative $R$-algebra with $1$ generated by elements $T_i:= T_{s_i}$, $i \in I$, and satisfying

\begin{itemize}
\item [(i)] the quadratic relations $(T_i+t)(T_i-t^{-1}) = T_i^2 +(t-t^{-1})T_i-1 = 0$ for all $i\in I$, and
\item [(ii)] the braid relations $T_iT_jT_i \cdots = T_jT_iT_j \cdots$
    ($m_{ij}$ factors on both sides of the equation) for any $i \neq j$ in $I$ with
    $s_{\al_i}s_{\al_j}$ of order $m_{ij}$ in $W$. If $m_{ij} = \infty$,
    there are no relations between $T_i$ and $T_j$.
\end{itemize}
\label{hecke}
\end{dfn}

\begin{thm}\label{thm:main}
Let $R$ be a (commutative unital) ring containing $\ZZ[t,t^{-1}]$ and let $u\in R$. Set $\mu_1 = u(t+t^{-1})$ and $\mu_2= -u^2$. Then if $F$ is the hyperbolic formal group law $F_{\mu_1,\mu_2}(u,v)=\frac{u+v-\mu_1 uv}{1+\mu_2 uv}$, the assignment $X_i \mapsto u(T_i+t)$ defines a morphism of $R$-algebras from the formal Demazure algebra $\DF$ to the Hecke algebra $H$ over $R$. If $u\in R^{\times}$, this morphism is an isomorphism.
\label{thm}
\end{thm}

\begin{proof}
Let $A$ be the free associative unital algebra on generators $X_i, i\in I$. There exists a unique (well-defined) algebra homomorphism $\psi:A \to H$ such that $\psi(X_i) = u(T_i + t)$. We check that $\psi$ annihilates the ideal generated by the relations of \ref{relationsDF} defining $\DF$. We have 
\begin{align*}
\psi(X_i^2- \mu_1 X_i) &= \psi(X_i)^2 - \mu_1 \psi(X_i) \\
&= u^2 T_i^2 + 2u^2tT_i + u^2t^2 -\mu_1 uT_i -\mu_1 u t \\
&= u^2((t^{-1}-t)T_i +1) + (2u^2 t-u^2(t+t^{-1})) T_i + u^2t^2-u^2t(t+t^{-1}) \text{ (by \ref{hecke}) } \\ 
&= (u^2t^{-1} - u^2 t+2u^2 t - u^2 t - u^2t^{-1})T_i + u^2 + u^2t^2-u^2t^2 - u^2 = 0
\end{align*}

We can check that $\psi(X_iX_j-X_jX_i) = 0$ since $T_iT_j=T_jT_i$ by \ref{hecke}. We also have 

\begin{align*}
\psi(X_{jij}-X_{iji} + \mu_2 X_j - \mu_2 X_i) &\stackrel{\eqref{hecke}}= u^3t(T_j^2-T_i^2) + u^2t^2(T_j-T_i) \\ &\quad + \mu_2 u(T_j+t)  - \mu_2u(T_i+t) \\
&= u^3t( (t^{-1}-t)T_j +1 - (t^{-1}-t)T_i - 1) \\ &\quad + u^2t^2(T_j-T_i)  + \mu_2u(T_j-T_i) + \mu_2ut-\mu_2ut \\
&= (u^3t(t^{-1}-t)+u^2t^2+\mu_2u)(T_j-T_i)\\ & =(u^3-u^3t^2+u^3t^2-u^3)(T_j-T_i)=0.
\end{align*}

Similarly, we get that $\psi(X_{jiji} - X_{ijij} + 2\mu_2(X_{ji} - X_{ij})) = 0$ and \[
\psi(X_{jijiji}-X_{ijijij} - 4\mu_2(X_{ijij} - X_{jiji}) - 3\mu_2^2(X_{ij} - X_{ji}))=0.\]

Since $\psi$ annihilates the relations defining $\DF$, it descends to a unital algebra homomorphism $\bar{\psi}\colon \DF \to H$ mapping $X_i\in \DF$ onto $\psi(X_i)\in H$. 

Moreover, if $u$ is invertible, $\bar{\psi}$ is surjective since the image of $\bar{\psi}$ contains the generators of $H$, i.e. $\bar{\psi}(u^{-1} X_i -t) = T_i \in \im(\bar{\psi})$. Also, we can find a surjective homomorphism in the other direction. Let $B$ be the free associative unital algebra on generators $T_i, i\in I$. There exists a unique (well-defined) algebra homomorphism $\phi\colon B \to \DF$ such that $\phi(T_i) = u^{-1} X_i -t$. We check that $\phi$ annihilates the ideal generated by the relations (i) and (ii) of \ref{hecke} defining~$H$. We have
\begin{align*}
\phi(T_i)^2 = ( u^{-1} X_i -t)^2 &= (u^{-2}X_i^2 - 2u^{-1}tX_i + t^2 \\ &= u^{-1}(t+t^{-1})X_i-2u^{-1}tX_i + t^2 \text{ (by \ref{specialelliptickappas})} \\ &= u^{-1}t^{-1}X_i - u^{-1}tX_i -1 + t^2 +1 \\ &= (t^{-1}-t)(u^{-1}X_i -t)+1 \\ &= (t^{-1}-t) \phi(T_i) +1.
\end{align*}
Therefore, $(\phi(T_i)+t)(\phi(T_i)-t^{-1}) = \phi(T_i)^2 +(t-t^{-1})\phi(T_i)-1 = 0$. We can also check that $\phi(T_iT_j-T_jT_i) = 0$ since $X_iX_j=X_jX_i$ by \ref{relationsDF}. We also have
\begin{align*}
\phi(T_{iji} - T_{jij}) &= u^{-3}(X_{iji}-X_{jij})- u^{-2}t(X_i^2-X_j^2) + u^{-1}t^2(X_i-X_j) \\ &= u^{-3}\mu_2(X_j-X_i) - u^{-2}t\mu_1(X_i-X_j) + u^{-1}t^2(X_i-X_j) \text{ (by \ref{specialelliptickappas})} \\ &= (-u^{-3}\mu_2 - u^{-2}t\mu_1 + u^{-1}t^2)(X_i-X_j) \\ &= (-u^{-3}u^2 - u^{-2}tu(t+t^{-1}) + u^{-1}t^2)(X_i-X_j) \\&= (-u^{-1} - u_{-1}t^2 - u^{-1} + u^{-1}t^2)(X_i-X_j) =0.
\end{align*}

Similarly, we get $\phi(T_{jiji} - T_{ijij})=0$ and $\phi(T_{jijiji} - T_{ijijij})=0$. Therefore, $\phi$ descends to a unital algebra homomorphism $\bar{\phi}\colon H \to \DF$, which is surjective and such that $\bar{\psi} \circ \bar{\phi}= \id_H$ and $\bar{\phi} \circ \bar{\psi}= \id_{\DF}$, hence, the isomorphism between $\DF$ and $H$. 
\end{proof}

As a consequence we obtain a similar result for formal affine Demazure
algebras.

\begin{cor}
Let $R$ be a (commutative unital) ring containing $\ZZ[t,t^{-1}]$, let $\Lambda$ be a formal Demazure lattice, and let $\Phi$ be a Kac-Moody root system. Let $\mu_1=t+t^{-1}$ and $\mu_2=-1$ and let $F_{\mu_1,\mu_2}(u,v)=\frac{u+v- (t+t^{-1})uv}{1-uv}$ be the hyperbolic formal group law. Let $\mathbf{H}$ be the $R$-algebra generated by the elements of $\RLF$ and by $T_i$, $i\in I$, subject to the relations (i) and (ii) of \ref{hecke}
and for all $\gamma \in \RLF$ and $i\in I$ \[\gamma T_i - T_i s_{\al_i}(\gamma)=(1-tx_{\al_i}) \Delta_{\al_i}(\gamma).\]
Then, the formal affine Demazure algebra $\ADF$ is isomorphic to $\mathbf{H}$ by a ring isomorphism preserving $\RLF$, in particular as an $R$-algebra. \label{cor1}
\end{cor}

\begin{proof} We proceed in the same way as in \ref{thm:main} with $u = 1$. We have a unital algebra homomorphism $\psi$ defined as the identity on $\RLF$ and mapping $X_i \mapsto T_i + t$. To show that $\psi$ annihilates the ideal generated by the relations \ref{relationsDF} defining $\ADF$ it remains to show that \[\psi(\gamma X_i - X_i s_{\al_i}(\gamma) - \Delta_{\al_i}(\gamma)) = 0\] for all $i \in I$ and all $\gamma \in \RLF$. We have 
\begin{align*} \psi( \gamma X_i - X_i s_{\al_i}(\gamma) - \Delta_{\al_i}(\gamma)) &= \gamma(T_i+t)-(T_i+t)s_{\al_i}(\gamma) - \Delta_{\al_i}(\gamma) \\&= \gamma T_i - T_i s_{\al_i}(\gamma) + t(\gamma-s_{\al_i}(\gamma)) - \Delta_{\al_i}(\gamma) \\ &= \gamma T_i - T_i s_{\al_i}(\gamma)-(1-tx_{\al_i})\Delta_{\al_i}(\gamma) =0.
\end{align*}
Therefore, we get a unital algebra homomorphism $\bar{\psi}\colon \ADF \to \mathbf{H}$ and we can show it is an isomorphism as in the proof of \ref{thm:main}. \qedhere
\end{proof}

\bibliographystyle{plain}

\end{document}